\documentclass{article}

\usepackage{cite}
\usepackage{amsmath,amssymb,amsfonts}
\usepackage{graphicx}
\usepackage{textcomp}

\usepackage{geometry}
\usepackage{xcolor}

\usepackage{cite}
\usepackage{amsmath,amssymb,amsfonts}
\usepackage{graphicx}
\usepackage{textcomp}

\usepackage{mathtools}
\usepackage{relsize}
\usepackage{tikz}

\usepackage{tabularx} 
\usepackage{graphicx} 
\usepackage{cite} 
\usepackage[final]{hyperref} 
\usepackage{amsmath} 
\usepackage{amssymb}  
\usepackage{mathtools, cuted,bm}
\usepackage{etoolbox}
\patchcmd{\proof}{\indent}{}{}{}
\usepackage{caption} 
\usepackage{tabularray}

\usepackage[justification=centering]{caption}

\newcommand{\twodots}{.\kern-0.1em.}
\newcommand{\myldots}{\kern-0.05em.\kern-0.01em.\kern-0.01em.\kern0.01em}

\usepackage{caption}
\usepackage{subcaption}
\usepackage{esvect}
\usepackage{lipsum, color}

\newcommand{\U}{\mathbb{U}}
\newcommand{\Y}{\mathbb{Y}}
\newcommand{\y}{\mathrm{y}}

\newcommand{\Hy}{H^{\mathrm{y}}}
\newcommand{\hy}{h^{\mathrm{y}}}

\newcommand{\R}{\mathbb{R}}
\newcommand{\I}{\mathbb{I}}

\newcommand{\fe}{\mathsf{f}}
\newcommand{\ve}{\mathsf{v}}

\usepackage{booktabs}

\usepackage{relsize}
\usepackage{tikz}
\graphicspath{{./Figures/}} 

\usepackage[margin=0.8cm]{caption}
\usepackage{pgfplots}
\pgfplotsset{compat=newest}
\usepackage{tabularx}
\pgfplotsset{plot coordinates/math parser=false}

\usepackage{ulem}
\usepackage{cancel}

\makeatletter
\def\munderbar#1{\underline{\sbox\tw@{$#1$}\dp\tw@\z@\box\tw@}}
\makeatother

\definecolor{olivegreen}{RGB}{128, 128, 0}

\usepackage{algorithm}
\usepackage{algpseudocode}

\usepackage{caption}

\newcommand\scalemath[2]{\scalebox{#1}{\mbox{\ensuremath{\displaystyle #2}}}}

\hypersetup{
	colorlinks=true,       
	linkcolor=black,        
	citecolor=black,        
	filecolor=magenta,     
	urlcolor=black         
}

\newtheorem{corollary}{Corollary}

\newtheorem{remark}{Remark}
\newtheorem{problem}{Problem}
\newtheorem{proposition}{Proposition}

\newtheorem{proof}{Proof}

\begin{document}
	\title{\LARGE \bf Learning Quasi-LPV Models and Robust Control Invariant Sets with Reduced Conservativeness}
	
	\author{Sampath Kumar Mulagaleti and Alberto Bemporad, \textit{Fellow, IEEE}
		\thanks{The authors are with the IMT School for Advanced Studies
			Lucca, Piazza San Francesco 19, 55100, Lucca, Italy (\url{s.mulagaleti,
				alberto.bemporad}@imtlucca.it). This work was supported by the European Research Council (ERC), Advanced Research Grant COMPACT (Grant Agreement No. 101141351).}
	}


	\newcounter{tempEquationCounter}
	\newcounter{thisEquationNumber}
	\newenvironment{floatEq}
	{\setcounter{thisEquationNumber}{\value{equation}}\addtocounter{equation}{1}
		\begin{figure*}[!t]
			\normalsize\setcounter{tempEquationCounter}{\value{equation}}
			\setcounter{equation}{\value{thisEquationNumber}}
		}
		{\setcounter{equation}{\value{tempEquationCounter}}
			\hrulefill\vspace*{4pt}
		\end{figure*}
		
	}
	\newenvironment{floatEq2}
	{\setcounter{thisEquationNumber}{\value{equation}}\addtocounter{equation}{1}
		\begin{figure*}[!t]
			\normalsize\setcounter{tempEquationCounter}{\value{equation}}
			\setcounter{equation}{\value{thisEquationNumber}}
		}
		{\setcounter{equation}{\value{tempEquationCounter}}
		\end{figure*}
	}

	\maketitle
    \thispagestyle{empty}
	
	\begin{abstract}
		We present an approach to identify a quasi Linear Parameter Varying (qLPV) model of a plant, with the qLPV model guaranteed to admit a robust control invariant (RCI) set. It builds upon the concurrent synthesis framework presented in \cite{Mulagaleti2024qlpv}, in which the requirement of existence of an RCI set is modeled as a control-oriented regularization.  Here, we reduce the conservativeness of the approach by bounding the qLPV system with an uncertain LTI system, which we derive using bound propagation approaches. The resulting regularization function is the optimal value of a nonlinear robust optimization problem that we solve via a differentiable algorithm. We numerically demonstrate the benefits of the proposed approach over two benchmark approaches.
	\end{abstract}

	\thispagestyle{empty}
	\pagestyle{empty}
	
	\section{Introduction}
	For synthesizing control schemes for nonlinear systems, the first step involves identifying a dynamic model via system identification techniques \cite{Ljung2020}. Typically, controllers use the predictive capabilities of these models to optimize performance. This approach, however, might prove inadequate in scenarios when the underlying system is subject to constraints, since the predicted model behavior might deviate from the system's response. Robust identification methods address this by capturing both a nominal behavior and prediction errors \cite{kosut1992set,Milanese2004,Lauricella2020}. Yet, decoupling identification from controller synthesis can lead to feasibility issues. This has motivated the development of concurrent synthesis approaches that co-identify an uncertain model and a robust controller \cite{Chen2018,Mulagaleti2022}. Some reinforcement learning approaches, e.g., \cite{Zanon2021}, can also be interpreted in such a framework.
	
	For control-oriented identification, an effective model class is quasi-Linear Parameter Varying (qLPV) systems~\cite{Lovera2013}. In qLPV systems, a.k.a. self-scheduled LPV systems, the dynamics are described by linear models that change over time as a function of a scheduling vector, whose values are generated by a nonlinear function of the model state. This has led to extensive research on identifying such models \cite{Rizvi2018,Verhoek2022}. Recently, in \cite{Mulagaleti2024qlpv} we introduced a concurrent synthesis framework that guarantees the existence of a robust control invariant (RCI) set \cite{Blanchini2008} for the identified 
	qLPV model based on a softmax scheduling function. This in turn ensures that a feedback controller exists for the plant generating the data. The RCI set was parameterized as a configuration-constrained polytope \cite{Villaneuva2024}, which provides a convenient representation for developing such frameworks.
	At the core of the concurrent synthesis framework introduced in \cite{Mulagaleti2024qlpv} is a control-oriented regularization function. This function is defined as the value function of a convex quadratic program (QP) and quantifies the size of the largest robust control invariant (RCI) set obtainable for a given parameterization. In fact, any model for which this regularization function admits a finite value becomes a candidate for controller synthesis. By exploiting the softmax parameterization, invariance is achieved by encapsulating the qLPV model within a linear time-invariant (LTI) model with multiplicative uncertainty.
    
	\textit{Contribution}: This paper improves the approach of \cite{Mulagaleti2024qlpv}, by identifying models that admit RCI sets with reduced conservativeness. We develop a new control-oriented regularization based on an uncertain linear system that encapsulates the qLPV system, and is tighter than one used in \cite{Mulagaleti2024qlpv}. The resulting regularization is optimal value of a nonlinear robust optimization problem, which is solved using a differentiable algorithm. The algorithm is then embedded into the concurrent synthesis framework. Through a numerical example, we demonstrate the benefits of the proposed approach. The paper is structured as follows:
    Section \ref{sec:definition} formalizes the concurrent synthesis problem, and recalls the approach of \cite{Mulagaleti2024qlpv} to formulate the regularization function; 
    Section \ref{sec:main_contribution} presents the novel regularization function formulation, along with the differentiable algorithm to evaluate it;
    Section \ref{sec:con_synthesis} presents an algorithm to solve the concurrent synthesis problem;
    Section \ref{sec:example} presents a numerical example, and validates it against benchmark approaches;
    Section \ref{sec:conclusions} summarizes the contribution and discusses future research directions.

	\textit{Notation:} The set $\mathbb{I}_a^b:=\{a,\cdots,b\}$ indicates indices between $a$ and $b$. Given sets $\mathcal{A},\mathcal{B} \subseteq \R^n$, $\mathcal{A}\oplus \mathcal{B}$ and $\mathcal{A} \ominus \mathcal{B}$ denote their Minkowski sum and difference respectively, with the sum denoted as $a \oplus \mathcal{B}$ if $\mathcal{A}=\{a\}$ is a singleton. Given $a \in \R^n$, $|a| \in \R^n$ denotes the element-wise absolute value vector, and the function $\mathrm{softmax}:\R^{n} \to \R^{n}$ is defined with components $e^{a_i}/\sum_{j=1}^n e^{a_j}$ for $i \in \mathbb{I}_1^n$. The set $\mathrm{CH}\{x_i,i\in \mathbb{I}_1^m\}$ is the convex hull of vectors $x_1,\cdots,x_m$. The symbol $\otimes$ denotes the Kronecker product.
	
	\section{Problem setup}
	\label{sec:definition}
	We have dataset $\mathcal{D}:=\{(u_t,y_t),t \in \mathbb{I}_0^{N-1}\}$ of input-output measurements from the nonlinear plant
	\begin{align}\label{eq:underlying_nonlinear}
		\mathbf{z}^+ = \mathbf{f}(\mathbf{z},u), &&
		y = \mathbf{g} (\mathbf{z}), 
	\end{align}
	where $\scalemath{0.98}{u \in \R^{n_u}, y \in \R^{n_y}, \mathbf{z} \in \R^{n_{\mathbf{z}}}}$ are the input, output, and state vectors. Assuming the functions $\mathbf{f}$ and $\mathbf{g}$, and state dimension $n_{\mathrm{z}}$, are unknown, we tackle the following problem:
	\begin{problem}
		\label{prob:basic_problem}
		\textit{Identify a model to predict the behaviour of~\eqref{eq:underlying_nonlinear} using $\mathcal{D}$, while ensuring that the model can be used to synthesize a feedback controller to regulate the plant output inside a set $Y$, i.e., $y \in \Y$, using control inputs $u \in \U$.} 
	\end{problem}
	
	While Problem~\ref{prob:basic_problem} can be tackled stage-wise, i.e., system identification with uncertainty characterization followed by robust controller synthesis, it is possible that the latter may be suboptimal/infeasible. We avoid this using a concurrent synthesis framework which integrates both phases.
    
	\subsection{Concurrent synthesis framework}
	Consider the qLPV model
	\begin{align}
		\label{eq:qLPV_model}
		x^+ = A(p(x))x+B(p(x))u, && \hat{y}=Cx
	\end{align}
	of~\eqref{eq:underlying_nonlinear}, where $x \in \R^{n_x}$ is the model state, and the matrix-valued functions $A(p)$ and $B(p)$ are parameterized as $A(p):=\sum_{i=1}^{n_p} p_i A_i$ and $B(p):=\sum_{i=1}^{n_p} p_i B_i$ respectively.
	The scheduling function $p : \R^{n_x} \to \R^{n_p}$ is defined as
	\begin{align}
		\label{eq:softmax}
		p(x;\theta):=\mathrm{softmax}\left([\mathcal{N}(x;\theta_1), \cdots,\mathcal{N}(x;\theta_{n_p})]^{\top}\right),
	\end{align}
	where each $\mathcal{N}(x;\theta_i)$ is a feedforward neural network (FNN) whose weights and biases are collected in the vector $\theta_i$. The parameterization in~\eqref{eq:softmax} enforces $p$ to belong to the simplex
	\begin{align*}
		\mathcal{P}:=\left\{ p \ \middle| \ \sum_{i=1}^{n_p} p_i = 1, 0\leq p \leq 1\right\}.
	\end{align*}
	\begin{remark}
		The results in the sequel can be extended to parameterizations $p(x,u)$, i.e., with dependency also on $u$. 
	\end{remark}
	
	The system identification problem computes model parameters $\mathbf{A}:=(A_1,\cdots,A_{n_p})$, $\mathbf{B}:=(B_1,\cdots,B_{n_p})$, $C$, and $\theta:=(\theta_1,\cdots,\theta_{n_p})$ by solving the optimization problem
	\begin{align}
		\label{eq:pure_sysID}
		&\min_{\mathbf{A},\mathbf{B},C,\theta,x_0} \ \frac{1}{N}\sum_{t=0}^{N-1} \|y_t - Cx_t\|_2^2 \\
		& \ \ \ \ \text{s.t.} \ \ \ x_{t+1} = A(p(x_t,\theta))x+B(p(x_t,\theta))u_t, \ t \in \mathbb{I}_0^{N-1}, \nonumber
	\end{align}
	in which also the initial state is optimized. Unfortunately, a constrained controller for a model obtained from~\eqref{eq:pure_sysID} is not guaranteed to exist, since the output $\hat{y}$ of~\eqref{eq:qLPV_model} might not match the plant output $y$ exactly. We ameliorate this using a \textit{control-oriented regularization}~\cite{Mulagaleti2024qlpv} based on the state-observer model
	\begin{align}
		\label{eq:state_observer}
		z^+ = A(p(z))z+B(p(z))u+L(p(z))w, 
	\end{align}
	where we parameterize the disturbance as $w:=y-Cz$ and the observer gain as $L(p):=\sum_{i=1}^{n_p} p_i L_i$.
	We now recall~\cite[Prop. 1, Prop. 2]{Mulagaleti2024qlpv} to derive an uncertain model based on~\eqref{eq:state_observer}.
	\begin{proposition}
		\label{prop:basic_observer}
		($i$) Suppose that the behavior of system~\eqref{eq:underlying_nonlinear} is described by the model
		\begin{align}
			\label{eq:underlying_fake}
			\hat{x}^+=A(p(\hat{x}))x+B(p(\hat{x}))u && y \in C\hat{x}+\mathbb{V},
		\end{align}
		for some $\mathbb{V} \subset \R^{n_y}$, and there exists some set $\mathcal{E} \subseteq \R^{n_x}$ that satisfies $\hat{x}_0 - z_0 \in \mathcal{E} \ \Rightarrow \ \hat{x}_t - z_t \in \mathcal{E}$ for all $t> 0$ 
		when~\eqref{eq:state_observer} and~\eqref{eq:underlying_fake} are excited by the same inputs. Defining
		\begin{align}
			\label{eq:W_definition}
			\mathbb{W}:=C\mathcal{E} \oplus \mathbb{V},
		\end{align}
        and denoting the set of states reached by \eqref{eq:state_observer} at time $t$ from some $z_0$ for all possible $w \in \mathbb{W}$ sequences and given input sequence as $Z_t$, it follows that $\hat{x}_0-z_0 \in \mathcal{E} \ \Rightarrow y_t \in CZ_t \oplus \mathbb{W}$ for all $t>0$ for all $v \in \mathbb{V}$ sequences, such that
        \begin{align}
			\label{eq:uncertain_model}
			\hspace{-7pt}
			\scalemath{0.97}{z^+ \hspace{-2pt} \in  \hspace{-2pt}  A(p(z))z+B(p(z))u \oplus L(p(z)) \mathbb{W}, 
				\ \ y \in Cz \oplus \mathbb{W} }
		\end{align}
        is a valid uncertain model of \eqref{eq:underlying_fake}; ($ii$) Suppose there exists a set $X \subseteq \R^{n_x}$ and control law $\mu_{\mathrm{c}}:X \to \U$ verifying%
        \begin{subequations}
        \label{eq:nonlinear_RCI}
        \begin{align}
           &  \hspace{-4pt} A(p(z))z+B(p(z))\mu(z)\oplus L(p(z)) \mathbb{W} \subseteq X,  \ \forall  z \in X, \\
            & \hspace{-4pt} CX \oplus \mathbb{W} \subseteq \Y.
        \end{align}
        \end{subequations}
        Then from any $z_0 \in X$ and $\hat{x}_0 \in z_0 \oplus  \mathcal{E}$, the control input $u_t = \mu_{\mathrm{c}}(z_t)$ ensures $z_t \in X$ and $y_t \in \Y$ for all $t\geq 0$. \hfill $\square$
	\end{proposition}

    The first result in Proposition~\ref{prop:basic_observer} states that the uncertain system~\eqref{eq:uncertain_model} can be \textit{forward simulated} to bound the plant output if \eqref{eq:W_definition} holds. This is a standard robust control approach using the observer invariant set $\mathcal{E}$ \cite{Mayne2006}. Through a suitable choice of $\mathbb{V}$, any system \eqref{eq:underlying_nonlinear} (potentially also including bounded noise) can be represented as \eqref{eq:underlying_fake}.
    The second result states that if the state of \eqref{eq:uncertain_model} can be persistently maintained inside $X$ using a control law $\mu_{\mathrm{c}}$, then it can be used to regulate the underlying plant output inside $Y$. As per \eqref{eq:nonlinear_RCI}, $X$ is an RCI set for \eqref{eq:uncertain_model}. In the sequel, we modify Problem \eqref{eq:pure_sysID} to $(i)$ Compute observer gains $\mathbf{L}:=(L_1,\cdots,L_{n_p})$, and $(ii)$ Guarantee the existence of a set $X$ verifying \eqref{eq:nonlinear_RCI}. Denoting $\Theta:=(\mathbf{A},\mathbf{B},C,\mathbf{L},\theta)$, we formulate the modified problem as
    \begin{align}
		\label{eq:concurrent_sysID}
		&\ \ \ \min_{\Theta,x_0} \ \ \frac{1}{N}\sum_{t=0}^{N-1} \|y_t - Cx_t\|_2^2 + \tau \mathbf{r}(\Theta)\\
		& \ \ \ \ \text{s.t.} \ \ \ x_{t+1} = A(p(x_t,\theta))x+B(p(x_t,\theta))u_t, \ t \in \mathbb{I}_0^{N-1}, \nonumber
	\end{align}
	where the \textit{control-oriented regularization} $\mathbf{r}(\Theta)$ is such that%
    \begin{subequations}
    \label{eq:r_requirement}
    \begin{align}
        \label{eq:r_requirement_1}
        &\mathbf{r}(\Theta)<\infty \ \ \Rightarrow \ \exists X \subseteq \R^{n_x} \ \text{verifying} \ \eqref{eq:nonlinear_RCI},     \\
        \label{eq:r_requirement_2}
        &\textit{Small } \mathbf{r}(\Theta) \ \Rightarrow \ \exists \textit{ Large } X \subseteq \R^{n_x} \ \text{verifying} \ \eqref{eq:nonlinear_RCI},
    \end{align}
    \end{subequations}
     and $\tau \geq 0$ is the regularization constant. While the requirement \eqref{eq:r_requirement_2} is informally stated, we will formalize it in the sequel. The goal of Problem \eqref{eq:concurrent_sysID} is to identify a model $\Theta$, while also maximizing the size of the corresponding RCI set. In this context, we say a model $\Theta_1$ is \textit{less conservative} that $\Theta_2$ if they admit RCI sets satisfying $\mathbf{r}(\Theta_1) < \mathbf{r}(\Theta_2)$.
    \subsection{Control-oriented regularization}
    We now recall the approach of \cite{Mulagaleti2024qlpv} which formulates $\mathbf{r}(\Theta)$ as inverse of the size of the largest RCI set admitted by $\Theta$.
    This formulation requires: ($i$) a characterization of the disturbance set $\mathbb{W}$; ($ii$) a description of the RCI set $X$; and ($iii$) a suitable measure to quantify the size of an RCI set. 

    \subsubsection{Characterization of $\mathbb{W}$} We utilize a simple data-driven characterization of $\mathbb{W}$, since the sets $\mathcal{E}$ and $\mathbb{V}$ required to compute it as \eqref{eq:W_definition} are unknown. Assuming access to a dataset $\mathcal{D}^{\mathrm{w}}:=\{(y_t,u_t),t \in \mathbb{I}_0^{N^{\mathrm{w}}-1}\}$ from \eqref{eq:underlying_nonlinear}, we simulate \eqref{eq:state_observer} from the origin using $\mathcal{D}^{\mathrm{w}}$ (with $w_t=y_t-Cz_t$), and denote the resulting state sequence as $\{z_t^{\mathrm{w}},t \in \mathbb{I}_0^{N^{\mathrm{w}}}\}$. Defining the sampled disturbances as $\mathcal{W}:=\{y_t - Cz_t^{\mathrm{w}} \mid t \in \mathbb{I}_0^{N^{\mathrm{w}}-1}\}$,
    we denote $\bar{\mathrm{w}}:=\max_{w \in \mathcal{W}} w$ and $\munderbar{\mathrm{w}}:=\min_{w \in \mathcal{W}} w$, along with $\mathrm{c}_{\mathrm{w}}:=0.5(\overline{\mathrm{w}}+\munderbar{\mathrm{w}})$ and $\epsilon_{\mathrm{w}}:=\frac{1}{2}(\overline{\mathrm{w}}-\munderbar{\mathrm{w}})$.
    Then, we characterize the set $\mathbb{W}$ for given $\Theta$ as
    \begin{align}
        \label{eq:W_characterization}
        \mathbb{W}:=\{w \ \mid \  |w-\mathrm{c}_{\mathrm{w}}| \leq \kappa \epsilon_{\mathrm{w}}\},
    \end{align}
    where $\kappa>0$ is a user-specified inflation parameter to account for finite data. Note that $\mathbb{W}$ is an inflated bounding box of the sampled disturbances $\mathcal{W}$ built using the dataset $\mathcal{D}^{\mathrm{w}}$. We refer to \cite[Prop. 4]{Mulagaleti2024qlpv} for lower-bounds on $\kappa>1$ to verify \eqref{eq:W_definition}.

    \subsubsection{Characterization of $X$} We work with polytopic RCI sets $X$ parameterized with given matrix $F \in \R^{\fe \times n_x}$ as
    \begin{align*}
		X \leftarrow X(q):=\{x \mid Fx \leq q\},
	\end{align*}
	and enforce configuration-constraints $\mathcal{C}$~\cite{Villaneuva2024} which dictate
	\begin{align*}
		q \in \mathcal{C}:=\{q \mid Eq\leq 0\} \Rightarrow  X(q)=\mathrm{CH}\{V_j q, j \in \mathbb{I}_1^{\ve}\},
	\end{align*}
	where $V_j \in \R^{n_x \times \fe}$ are vertex maps. We refer to~\cite{Villaneuva2024} for details about computing $E$ and $V:=(V_1,\cdots,V_{\ve})$ given $F$.  
    To enforce the RCI constraints in \eqref{eq:nonlinear_RCI}, the approach of \cite{Mulagaleti2024qlpv} exploits the parameterization of the scheduling function in \eqref{eq:softmax}, which implies $p(z) \in \mathcal{P}$ for all $z \in \R^{n_x}$ such that
    \begin{align}
    \label{eq:multiplicative_uncertainty_full}
    \hspace{-5pt}
    (A(p),B(p),L(p)) \in \Delta:= \mathrm{CH}\{(A_i,B_i,L_i), i \in \mathbb{I}_1^{n_p}\}
    \end{align}
    for all $p$. Then, an RCI set for the uncertain linear system
    \begin{align}
    \label{eq:original_uncertain}
        z^+ = Az+Bu+Lw, && (A,B,L) \in \Delta, w \in \mathbb{W}
    \end{align}
    is RCI for uncertain system \eqref{eq:uncertain_model}. 
    To enforce that $X(q)$ is an RCI set for \eqref{eq:original_uncertain}, we use the following result from \cite{Villaneuva2024}, where we denote $U_j:=e_j^{\top} \otimes \I_{n_u}$ for $j \in \mathbb{I}_1^{\ve}$, and assume that the output constraint set $\Y:=\{y \mid \Hy y\leq \hy\}$.
    \begin{proposition}
    \label{prop:RCI_proposition}
       The set $X(q)$ is an RCI set for \eqref{eq:original_uncertain} with constraints $Cz \oplus \mathbb{W} \subseteq \Y$ and $u \in \U$ if there exists some $v \in \R^{\ve \cdot n_u}$ such that $(q,v) \in \mathbb{S}$, where we define
       \begin{align}
           \label{eq:S_definition}
           \scalemath{0.98}{\mathbb{S}:=\left\{\begin{pmatrix} q \\ v \end{pmatrix} \middle| \ 
           \begin{matrix*}[l] 
            \forall (i,j) \in \mathbb{I}_1^{n_p} \times \mathbb{I}_1^{\ve}, Eq \leq 0, U_j v \in \U,\\
            F(A_i V_j q + B_i U_j v)+d_i \leq q, \\
            \Hy (CV_jq+\mathrm{c}_{\mathrm{w}})+\kappa |\Hy| \epsilon^{\mathrm{w}} \leq \hy
           \end{matrix*} \right\},}
       \end{align} 
       with $d_i:=FL_i \mathrm{c}_{\mathrm{w}}+\kappa |FL_i| \epsilon_{\mathrm{w}}$ for all $i \in \mathbb{I}_1^{n_p}$. \hfill $\square$
    \end{proposition}

    Essentially, Proposition \ref{prop:RCI_proposition} together with \eqref{eq:multiplicative_uncertainty_full} state that
    \begin{align}
    \label{eq:desirable_property}
        \exists \ v : (q,v) \in \mathbb{S} \ \Rightarrow \ X(q) \ \text{satisfies} \ \eqref{eq:nonlinear_RCI},
    \end{align}
    with $\mu_{\mathrm{c}}(z)$ being a vertex control law defined by $v$ \cite{Villaneuva2024}.

    \subsubsection{Size of the RCI set}
    We define the size of an RCI set based on the ability of the system to perform safe output tracking. Denoting the vertices of $\mathbb{Y}$ as $\{\y_k, k \in \mathbb{I}_1^{v_y}\}$ and the mean values of the matrices $(A_i,B_i),i \in \mathbb{I}_1^{n_p}$ as $(\overline{A},\overline{B})$, 
    the size of $X(q)$ is modeled as
    \begin{align}
        \mathrm{d}(\overline{A},\overline{B},C,q):=&\min_{\mathrm{z},\mathrm{u}} \ \ \sum_{k=1}^{v_y} \sum_{t=1}^{M} \|\y_k-Cz^k_t\|_2^2 \label{eq:quality_function}\\
        & \  \ \text{s.t.} \ \ \scalemath{0.95}{z_0^k = 0, \ z_{t+1}^k = \overline{A} z^k_t + \overline{B} u^k_t,  \ u^k_t \in \U,} \nonumber \\
        & \quad \ \ \ \ \ \scalemath{0.95}{Fz^k_t \leq q, \  \forall (k,t) \in \mathbb{I}_1^{v_y} \times \mathbb{I}_0^{M-1}}, \nonumber
    \end{align}
    where $(\mathrm{z},\mathrm{u})$ denote the state and input trajectories. The value $\mathrm{d}(\overline{A},\overline{B},C,q)$ captures how close the output $y=Cz$ of the nominal system $z^+=\overline{A}z+\overline{B}u$ can be driven from the origin to the vertices of $\mathbb{Y}$ in $M$-steps while belonging inside $X(q)$. Note that for any feasible vectors $q_1$ and $q_2$, the inequality
    \begin{align}
    \label{eq:size_inequality}
        X(q_1) \subseteq X(q_2) \ \Rightarrow \ \mathrm{d}(\overline{A},\overline{B},C,q_2) \leq \mathrm{d}(\overline{A},\overline{B},C,q_1)
    \end{align}
    holds since the optimizers of the latter problem are feasible for the former.
    Using the ingredients in \eqref{eq:W_characterization}, \eqref{eq:S_definition} and \eqref{eq:quality_function}, the approach of \cite{Mulagaleti2024qlpv} models the control-oriented regularization function $\mathrm{r}(\Theta)$ in Problem \eqref{eq:concurrent_sysID} as the value of the QP
    \begin{align}
    \label{eq:old_r}
        \mathrm{r}(\Theta):=\inf_{(q,v) \in \mathbb{S}} \ \mathrm{d}(\overline{A},\overline{B},C,q).
    \end{align}
    As per \eqref{eq:desirable_property} and \eqref{eq:size_inequality}, $\mathbf{r}(\Theta)$ satisfies \eqref{eq:r_requirement}, making it a suitable control-oriented regularization. 
    However, an RCI set $X(q)$ obtained by solving \eqref{eq:old_r} might be unnecessarily \textit{small}, because the uncertain LTI system \eqref{eq:original_uncertain} encapsulating \eqref{eq:uncertain_model} might be too conservative. In the next section, we derive an  uncertain LTI approximation that is less conservative, and formulate a corresponding control-oriented regularization.
    \begin{remark}
    The choice of function $\mathrm{d}$ should reflect the goals of the control design procedure, e.g., ($i$) regulation around an output setpoint, with $\mathrm{d}$ modeled to minimize the volume of the RCI set about that point \cite{Trodden2016};  ($ii$) tuned tube-based model predictive control schemes \cite{Badalamenti2024} for stabilization, etc. Future research can focus on deriving such formulations.
    \end{remark}

    \section{Control-oriented regularization with reduced conservativeness}
    \label{sec:main_contribution}
    While the uncertain LTI system \eqref{eq:original_uncertain} encapsulates the nonlinear system \eqref{eq:uncertain_model}, it is possible that the scheduling function $p(z)$ with $z \in X(q)$ does not cover the entire set $\mathcal{P}$, such that $(A(p(z)),B(p(z)),L(p(z)))$ only evolves in some $\tilde{\Delta} \subseteq \Delta$. We exploit this observation to derive a tightened multiplicative uncertainty $\tilde{\Delta}$, and present an approach to compute an RCI set for \eqref{eq:original_uncertain} with $\Delta$ replaced by $\tilde{\Delta}$.
    \subsection{RCI sets with reduced conservativeness}
    We now present new conditions for a set $X(q)$ to be an RCI set for \eqref{eq:uncertain_model}, which are based on an LTI system with tightened multiplicative uncertainty $\tilde{\Delta} \subseteq \Delta$.
    \begin{proposition}
		\label{prop:BP_result}
		Suppose the scheduling variable satisfies
		\begin{align}
			\label{eq:a_lb_p}
			0 \leq a_i \leq p_i(z), && \forall z \in X(q), i \in \mathbb{I}_1^{n_p}
		\end{align}
		for given $q \in \mathcal{C}$. If $X(q)$ is RCI for the uncertain LTI system
		\begin{align}
			\label{eq:uncertain_tighter}
			\hspace{-6pt} z^+ = Az+Bu + Lw, && (A,B,L) \in \tilde{\Delta}, w \in \mathbb{W},
		\end{align}
		where $\tilde{\Delta}:=\mathrm{CH}\{(\tilde{A}_i,\tilde{B}_i,\tilde{L}_i),i \in \mathbb{I}_1^{n_p}\}$ is defined with
		\begin{align}
			(\tilde{A}_i,\tilde{B}_i,\tilde{L}_i):=\scalemath{0.8}{\left(1-\sum_{j=1}^{n_p}a_j\right)}(A_i,B_i,L_i) + \scalemath{0.8}{\sum_{j=1}^{n_p}} a_j (A_j,B_j,L_j)  \label{eq:tilde_definitions}
		\end{align}
		then $X(q)$ verifies the inclusion in~\eqref{eq:nonlinear_RCI}.
	\end{proposition}
	\begin{proof}
		The proof follows if for all $z \in X(q)$, we have $(A(p(z)),B(p(z)),L(p(z))) \in \tilde{\Delta}$. To show this, we define 
		\begin{align*}
			\tilde{p}_i(z) = \cfrac{p_i(z)-a_i}{1-\sum_{j=1}^{n_p}a_j}, && \forall i \in \mathbb{I}_1^{n_p}.
		\end{align*}
		Since $\sum_{j=1}^{n_p} a_j \leq \sum_{j=1}^{n_p} p_j(z) \leq 1$ holds, we have $\tilde{p}(z) \in \mathcal{P}$ if $z \in X(q)$. Then, $A(p(z))$ can be written for $z \in X(q)$ as
		\begin{align*}
			A(p(z)) &= \scalemath{0.8}{\sum_{i=1}^{n_p} \left(\tilde{p}_i(z)+a_i-\tilde{p}_i(z) \sum_{j=1}^{n_p}a_j \right)} A_i \\
			&\hspace{-25pt} =\scalemath{0.8}{\sum_{i=1}^{n_p} \left(\tilde{p}_i(z)+a_i\sum_{j=1}^{n_p} \tilde{p}_j(z)-\sum_{j=1}^{n_p} a_j \tilde{p}_i(z) \right)}A_i \\
			&\hspace{-25pt}  = \scalemath{0.8}{\sum_{i=1}^{n_p} \left(1-\sum_{j=1}^{n_p} a_j \right)\tilde{p}_i(z)} A_i + \scalemath{0.8}{\sum_{i=1}^{n_p} \sum_{j=1}^{n_p} a_j \tilde{p}_i(z)}A_i  = \sum_{i=1}^{n_p} \tilde{p}_i(z) \tilde{A}_i,
		\end{align*}
		where the second equality follows from $\tilde{p}(z) \in \mathcal{P}$, and the third by interchanging the summation order. Hence, we have that $(A(p(z)),B(p(z)),L(p(z))) \in \tilde{\Delta}$ if $z \in X(q)$.
	\end{proof}
    
    The next result derives a set similar to $\mathbb{S}$ for \eqref{eq:uncertain_tighter}.
    \begin{corollary}
    \label{corr:RCI_proposition_new}
       The set $X(q)$ is an RCI set for \eqref{eq:uncertain_tighter} with constraints $Cz \oplus \mathbb{W} \subseteq \Y$ and $u \in \U$ if there exists some $v \in \R^{\ve \cdot n_u}$ such that $(q,v) \in \tilde{\mathbb{S}}$, where we define
       \begin{align*}
           \tilde{\mathbb{S}}:=\scalemath{0.98}{\left\{\begin{pmatrix} q \\ v \end{pmatrix} \middle|  \
           \begin{matrix*}[l] 
           \forall (i,j) \in \mathbb{I}_1^{n_p} \times \mathbb{I}_1^{\ve},  \ \exists a \in [0,p(z)], \forall z \in X(q),   \\
            F(\tilde{A}_i(a) V_j q + \tilde{B}_i(a) U_j v)+\tilde{d}_i(a) \leq q, \\
            \Hy (CV_jq+\mathrm{c}_{\mathrm{w}})+\kappa |\Hy| \epsilon^{\mathrm{w}} \leq \hy,  \\
            U_j v \in \U, Eq \leq 0 
           \end{matrix*} \right\}},
       \end{align*} 
       with $\tilde{d}_i(a):=F\tilde{L}_i(a) \mathrm{c}_{\mathrm{w}}+\kappa |F\tilde{L}_i(a)| \epsilon_{\mathrm{w}}$ for $i \in \mathbb{I}_1^{n_p}$, and the matrices $(\tilde{A}_i,\tilde{B}_i,\tilde{L}_i)$ depend on $a_i$ through \eqref{eq:tilde_definitions}. \hfill $\square$
    \end{corollary}
    \begin{proof}
        The proof follows from Proposition \ref{prop:RCI_proposition}, after observing that $\mathbb{S}$ is RCI for \eqref{eq:uncertain_tighter} for any $a$ verifying \eqref{eq:a_lb_p}, with $(A_i,B_i,L_i)$ replaced by $(\tilde{A}_i(a),\tilde{B}_i(a),\tilde{L}_i(a))$.
    \end{proof}
    
    Using $\tilde{\mathbb{S}}$, we formulate the regularization function as
    \begin{align}
    \label{eq:new_r}
        \mathrm{r}(\Theta):=\inf_{(q,v) \in \tilde{\mathbb{S}}} \ \mathrm{d}(\overline{A},\overline{B},C,q),
    \end{align}
    where $(\overline{A},\overline{B})$ are mean values of $(\tilde{A}_i(a),\tilde{B}_i(a)),i \in \mathbb{I}_1^{n_p}$. We always have $\mathbb{S} \subseteq \tilde{\mathbb{S}}$, since $\tilde{\mathbb{S}}$ formulated with feasible value $a=0$ equals $\mathbb{S}$. Hence, \eqref{eq:new_r} is less conservative than \eqref{eq:old_r}. Unfortunately, $\tilde{\mathbb{S}}$ is no longer a polytope since it includes the robust nonlinear constraint \eqref{eq:a_lb_p}, such that Problem \eqref{eq:new_r} is a nonlinear robust optimization problem instead of a QP. We now develop an algorithm to solve Problem \eqref{eq:new_r}. 

    \subsection{Solving Problem \eqref{eq:new_r}}
    We use the following observation to solve Problem \eqref{eq:new_r}: Given sets $X_1$ and $X_2$ such that $X_1 \subseteq X_2$, if $p_i(z) \geq a_i$ holds for all $z \in X_2$, then $p_i(z) \geq a_i$ follows for all $z \in X_1$. To exploit this observation, given $q \in \mathcal{C}$ we define the set
    \begin{align}
		\label{eq:bounding_box}
		\mathbb{B}(q):=\left\{x : |x_i - \mu_i(q)| \leq \sigma_i(q)+\zeta, i \in \mathbb{I}_1^{n_x}\right\},
	\end{align}
	where we define $\mu_i(q):=\frac{1}{2}(\max_{j} V_{j,i}q+\min_{j} V_{j,i}q)$ and $\sigma_i(q):=\frac{1}{2}(\max_{j} V_{j,i}q-\min_{j} V_{j,i}q)$,
	with $V_{j,i}$ denoting row $i$ of $V_j$, and $\zeta > 0$. The set $\mathbb{B}(q)$ is a bounding box of $X(q)$ inflated by $\zeta>0$. Then, we define the polytope
    \begin{align*}
           \hat{\mathbb{S}}(\tilde{q},a):=\scalemath{0.98}{\left\{\begin{pmatrix} q \\ v \end{pmatrix} \middle|  \
           \begin{matrix*}[l] 
           \forall (i,j) \in \mathbb{I}_1^{n_p} \times \mathbb{I}_1^{\ve}, \ X(q) \subseteq \mathbb{B}(\tilde{q}), \\
            F(\tilde{A}_i(a) V_j q + \tilde{B}_i(a) U_j v)+\tilde{d}_i(a) \leq q, \\
            \Hy (CV_jq+\mathrm{c}_{\mathrm{w}})+\kappa |\Hy| \epsilon^{\mathrm{w}} \leq \hy,  \\
            U_j v \in \U, Eq \leq 0
           \end{matrix*} \right\}},
       \end{align*} 
       where the parameters $a \in \R^{n_p}$ depend on $\tilde{q}$ as
       \begin{align}
           \label{eq:a_lb_2}
			0 \leq a_i \leq p_i(z), && \forall z \in \mathbb{B}(\tilde{q}), i \in \mathbb{I}_1^{n_p}.
       \end{align}
        The polytope $\hat{\mathbb{S}}(\tilde{q},a)$ is formulated by replacing condition \eqref{eq:a_lb_p} over $a$ in $\tilde{\mathbb{S}}$ with $X(q) \subseteq \mathbb{B}(\tilde{q})$. Since \eqref{eq:a_lb_2} implies $a \in [0,p(z)]$ for all $z \in X(q)$, it follows that $\hat{\mathbb{S}}(\tilde{q},a) \subseteq \tilde{\mathbb{S}}$. If $\tilde{q}$ is high enough such that the constraint $X(q) \subseteq \mathbb{B}(\tilde{q})$ is inactive for all feasible $q$, the inclusions $\mathbb{S} = \hat{\mathbb{S}}(\tilde{q},0) \subseteq \hat{\mathbb{S}}(\tilde{q},a) \subseteq \tilde{\mathbb{S}} $
        follow for any $a$ verifying \eqref{eq:a_lb_2}, which implies that an RCI set computed by optimizing over $\hat{\mathbb{S}}(\tilde{q},a)$ is less conservative than that obtained by optimizing over $\mathbb{S}$.
        Utilizing $\hat{\mathbb{S}}(\tilde{q},a)$, we formulate Algorithm \ref{alg:RCI} to solve Problem \eqref{eq:new_r}. Starting from some $q_0 \in \mathcal{C}$, we compute bounds $a$ verifying \eqref{eq:a_lb_2} using a $\mathrm{BoundProp}$ methodology presented next. Then, we compute $(\tilde{A}_i(a),\tilde{B}_i(a),\tilde{L}_i(a))$ used to define a QP in Step 3 that optimizes over $\hat{\mathbb{S}}(q_k,a)$ constructed using $\mathbb{B}(q_k)$. Then, updating $q_k$ to $q_{k+1}$, we repeat the procedure for $\hat{k}$ number of steps. The key feature of Algorithm \ref{alg:RCI} is that the output $\mathbf{r}_{\hat{k}}$ is differentiable in $\Theta$, such that it can be plugged into a gradient-based solver to tackle Problem \eqref{eq:concurrent_sysID}. Future research can study the theoretical properties of this algorithm.
    \begin{algorithm}[t]
	\caption{Evaluate $\mathbf{r}(\Theta;q_0)$}
	\label{alg:RCI}
	\begin{algorithmic}[1]
		\Require $\Theta$, $(\mathrm{c}_{\mathrm{w}},\epsilon_{\mathrm{w}})$, $q_0 \in \mathcal{C}$, $\hat{k}>0$ 
		\State $\textbf{For} \ k = 0,1,2,\ldots,\hat{k}-1 \ \textbf{do}$
		\State $\quad \scalemath{0.9}{a \gets \mathrm{BoundProp}(q_k,\Theta)}$, evaluate $\scalemath{0.9}{(\tilde{A}_i(a),\tilde{B}_i(a),\tilde{L}_i(a))}$
		\State $\quad q_{k+1},v_{k+1} \gets \arg\inf_{(q,v) \in \hat{\mathbb{S}}(q_k,a)} \mathrm{d}(\overline{A},\overline{B},C,q) $
		\State $\quad \mathbf{r}_{k+1} \gets \mathrm{d}(\overline{A},\overline{B},C,q_{k+1})$
		\State \Return $\mathbf{r}_{\hat{k}}, q_{\hat{k}}$
	\end{algorithmic}
\end{algorithm}
\subsection{Interval bound propagation}
We use interval bound propagation (IBP)~\cite{gowal2018effectiveness} to compute $a$ verifying \eqref{eq:a_lb_2} in Step 3 of Algorithm \ref{alg:RCI}, assuming that the FNNs $\mathcal{N}(\cdot;\theta_i)$ are defined using monotonic activations.
While IBP computes $a$ exploiting this monotonicity, alternative verification approaches can be used to handle nonmonotonic activations,  see, e.g.,~\cite{wang2021betacrownefficientboundpropagation, fazlyab2021safetyverificationrobustnessanalysis,Paulsen2022}. However, they might be computationally expensive. The IBP approach, while limited to monotonic activations, is computationally inexpensive. It is formed by observing that the interval $z \in [\munderbar{z},\bar{z}]$ projected through the linear map $h(z) = Wz+b$ results in $h(z) \in [W\mu+b-|W|\Sigma, W\mu+b+|W|\Sigma]$
with $\mu = 0.5(\bar{z}+\munderbar{z})$ and $\Sigma = 0.5(\bar{z}-\munderbar{z})$, while through a monotonically increasing nonlinear map $h(z)$ results in $h(z) \in [h(\munderbar{z}),h(\bar{z})].$
Defining $\mathbf{N}(z;\theta_i):=e^{\mathcal{N}(z;\theta_i)}$, a composition of the propagations can be used to compute the bounds
\begin{align}
	\label{eq:before_softmax}
	\mathbf{N}(z;\theta_i) \in [\munderbar{\mathbf{N}}_i,\overline{\mathbf{N}}_i], && \forall z \in \mathbb{B}(q).
\end{align}
\begin{proposition}
	\label{prop:softmax_BP}
    Given $q \in \mathcal{C}$, suppose~\eqref{eq:before_softmax} holds. Then,
	\begin{align*}
		a_i:=\cfrac{\munderbar{\mathbf{N}}_i}{\left(\munderbar{\mathbf{N}}_i + \sum_{j \in \mathbb{I}_1^{n_p} \setminus i} \overline{\mathbf{N}}_j\right)}
	\end{align*}
    verifies \eqref{eq:a_lb_2} with $\tilde{q}=q$ for all $i \in \mathbb{I}_1^{n_p}$.
\end{proposition}
\begin{proof}
	We observe that~\eqref{eq:before_softmax} satisfies the inequalities $0 \leq \munderbar{\mathbf{N}}_i \leq \overline{\mathbf{N}}_i$, and the optimizer of $\min_{x,y} \ \frac{x}{x+y} \ \text{s.t.} \ (x,y) \in [\munderbar{x},\bar{x}] \times [\munderbar{y},\bar{y}]$
	with $0\leq \munderbar{x} \leq \bar{x}$ and $0 \leq \munderbar{y}\leq \bar{y}$ is $(\munderbar{x},\bar{y})$.
\end{proof}

\begin{algorithm}[t]
	\caption{Solve Problem~\eqref{eq:concurrent_sysID}}
	\label{alg:concurrent}
	\begin{algorithmic}[1]
		\Require $(\mathcal{D}, \mathcal{D}^{\mathrm{w}})$, $(F,V,E)$, $\kappa >1$, $\zeta>0$,  $\hat{k},\hat{l}>0$, Initial parameters $\Theta_0,x_{0,0}$, $q_0 \in \mathcal{C}$.
		\State $\textbf{For} \ l = 0,1,2,\ldots,\hat{l}-1 \ \textbf{do}$
		\State $\quad \Theta_{l+1},x_{0,l+1} \leftarrow \mathrm{Optimizer}(\nabla \mathcal{J}(\Theta_l,x_{0,l};q_l))$
		\State $\quad \mathbf{r}_{l+1},q_{l+1} \leftarrow \mathbf{r}(\Theta_{l+1};q_l)$ 
		\State \Return $\Theta_l$
	\end{algorithmic}
\end{algorithm}

\section{Concurrent synthesis algorithm}
\label{sec:con_synthesis}
We now develop Algorithm~\ref{alg:concurrent} to solve Problem~\eqref{eq:concurrent_sysID}, in which we utilize the fact that Algorithm~\ref{alg:RCI} is parametric in $q_0$. Towards its development, for a given $\tilde{q} \in \mathcal{C}$, we define
\begin{align*}
	\scalemath{0.9}{\mathcal{J}(\Theta,x_0;\tilde{q}):=\frac{1}{N}\sum_{t=0}^{N-1} \|y_t - Cx_t\|_2^2 + \tau \mathbf{r}(\Theta;\tilde{q}),}
\end{align*}
where $\mathbf{r}(\Theta;\tilde{q})$ is evaluated using Algorithm~\ref{alg:RCI}. We use this function as an alias for the objective of Problem~\eqref{eq:concurrent_sysID}. 
We initialize the algorithm with $q_0 \in \mathcal{C}$ being the optimizer of Problem \eqref{eq:old_r} with $\Theta=\Theta_0$. 
In Step 3, with the last updated RCI set parameter $q_l$, we compute the gradients of $\mathcal{J}(\Theta,x_0;q_l)$, that we pass to an optimizer such as Adam~\cite{kingma2017} to compute $\Theta_{l+1}$ and $x_{0,l+1}$. Then, we update $q_{l+1}$ using the updated model parameters $\Theta_{l+1}$ starting from $q_l$ using Algorithm~\ref{alg:RCI}. To compute the initial $\Theta_0$ and configuration triplet $(F,V,E)$, we use the approach in~\cite[Section IV]{Mulagaleti2024qlpv}.
\section{Numerical example}
\label{sec:example}
We consider data from an oscillator with dynamics given by $1.5 \ddot{y}+\dot{y}+y+1000 y^3=u,$
with input $u \in [-0.5,0.5]$[N] the force applied, and output $y$ [m] the position. While we focus on the benefits of using the approach of this paper for concurrent synthesis, we refer to \cite{Mulagaleti2024qlpv} for insights regarding the qLPV parameterization. We implement Algorithm \ref{alg:concurrent} using \url{jax} \cite{jax2018github}, and
utilize the differentiable QP solver \url{qpax} \cite{tracy2024differentiability} to implement of Algorithm \ref{alg:RCI}. We evaluate the quality of a model using the Best Fit Rate (BFR) score~\cite[Section 3.3]{Bem25}\footnote{Code available on \url{github.com/samku/Con-qLPV}}. 
\subsubsection{Concurrent identification}
Using randomly sampled inputs in $[-0.5,0.5]$, we build the training dataset $\mathcal{D}$ with $10000$ points, disturbance dataset $\mathcal{D}^{\mathrm{w}}$ with $2000$ points, and test dataset with $10000$ points sampled at $0.1$s time intervals. We parameterize~\eqref{eq:qLPV_model} with $n_x=2$, $n_p=6$, and a single-hidden-layer FNN $\mathcal{N}(x;\theta_i)$ with $3$ monotonic activation units $\mathrm{elu}(x)+1$. We compute an initial qLPV model parameterized as \eqref{eq:qLPV_model} using the \url{jax-sysid}~\cite{Bem25} toolbox on $\mathcal{D}$, which achieves BFR scores of $85.887$ on $\mathcal{D}$, $85.97038$ on $\mathcal{D}^{\mathrm{w}}$, and $86.8900$ on the test set.  We use this model to define $\Theta_0$, in which we set each observer gain $L_i=0$. Using \cite[Proposition 5]{Mulagaleti2024qlpv}, we compute a feasible triplet \((F,V,E)\) with $\fe=\ve=4$ for $X(q)$. We select $M=5$ in \eqref{eq:quality_function}, $\tau=0.0005$ in \eqref{eq:concurrent_sysID}, $\kappa=1.01$ in \eqref{eq:W_characterization}, and compute $q_0$ by solving Problem \eqref{eq:old_r} to initialize Algorithm~\ref{alg:concurrent}. 
We consider the following benchmarks: ($i$) \textbf{\textit{Sequential model-based synthesis}}: We compute the parameters of \eqref{eq:qLPV_model} starting from $\Theta_0$ using the \url{jax-sysid} toolbox, then utilize $\mathcal{D}^{\mathrm{w}}$ to identify $\mathbb{W}$ and compute the maximal RCI set \cite{Bertsekas1972} for \eqref{eq:original_uncertain} inside the tightened output constraint set $Cx \in \mathbb{Y} \ominus \mathbb{W}$. Then, selecting $(F,q)$ such that $\{x \mid Fx \leq q\}$ is the maximal RCI set, we evaluate $\mathrm{d}_{\mathrm{seq}}:=\mathrm{d}(\overline{A},\overline{B},C,q)$ defined in \eqref{eq:quality_function}. The identified model achieves BFR scores of $90.7979$ on $\mathcal{D}$, $92.3137$ on $\mathcal{D}^{\mathrm{w}}$, and $90.8584$ on the test set, along with $\mathrm{d}_{\mathrm{seq}}=22.6924$.
($ii$) \textbf{\textit{Baseline concurrent synthesis}}: We follow the approach of \cite{Mulagaleti2024qlpv} to formulate the control-oriented regularization as \eqref{eq:old_r}. At the solution, we re-solve \eqref{eq:old_r}, and denote its optimal value as $\mathrm{d}_{\mathrm{base}}$. The identified model achieves BFR scores of $90.6734$ on $\mathcal{D}$, $92.0751$ on $\mathcal{D}^{\mathrm{w}}$, and $90.6593$ on the test set, along with $\mathrm{d}_{\mathrm{base}}=5.0936$ indicating the benefits of utilizing concurrent synthesis for reducing RCI set conservativeness. 

To compare against these benchmarks, we simulate Algorithm \ref{alg:concurrent} with $\zeta$ uniformly spaced in $[0.01,0.1]$, and fix $\hat{k}=1$ such that we perform one iteration of Algorithm \ref{alg:RCI} per iteration of Algorithm \ref{alg:concurrent}. We report that while $\hat{k}>1$ can be chosen to simulate Algorithm \ref{alg:RCI}, it often results in Algorithm \ref{alg:concurrent} converging to suboptimal points. A study of escaping such minima is a subject of future study. Using the model parameters $\Theta_{\zeta}$ computed by Algorithm \ref{alg:concurrent} for given $\zeta$, we recompute the largest RCI set by solving Problem \eqref{eq:new_r} with Algorithm \ref{alg:RCI} with $q_0$ as the optimizer of Problem \eqref{eq:old_r}. We denote the output of Algorithm \ref{alg:RCI} as $\mathrm{d}_{\zeta}=\mathbf{r}_{\hat{k}}$, with $\hat{k}=200$. We report no reduction in $\mathbf{r}$ when Algorithm \ref{alg:RCI} is applied to the models identified using the benchmark approaches, since these models were not optimized for bound propagation.
In Figure \ref{fig:zeta_variation}, we plot $\mathrm{d}_{\zeta}$ for different values of $\zeta$. We observe a reduction in conservativeness, indicated by smaller values than $\mathrm{d}_{\mathrm{seq}}$ and $\mathrm{d}_{\mathrm{base}}$, validating our approach to compute models that admit RCI sets with reduced conservativeness. We report that for all models $\Theta_{\zeta}$, we obtain BFR scores in the interval $[90.7303,90.7371]$ over $\mathcal{D}$, $[92.4489,92.4633]$ over $\mathcal{D}^{\mathrm{w}}$, and $[90.8314,90.8553]$ over the test set, indicating that a reduction in conservativeness without degrading the model quality. In Figure \ref{fig:alg_1}, we show the value of $\mathbf{r}_k$ over iterations of Algorithm \ref{alg:RCI} with $\Theta=\Theta_{0.07}$ and $\zeta=0.07$, in which we observe monotonic convergence.
\begin{figure}[t]
	\centering
	\includegraphics[width=0.85\linewidth, clip=true]{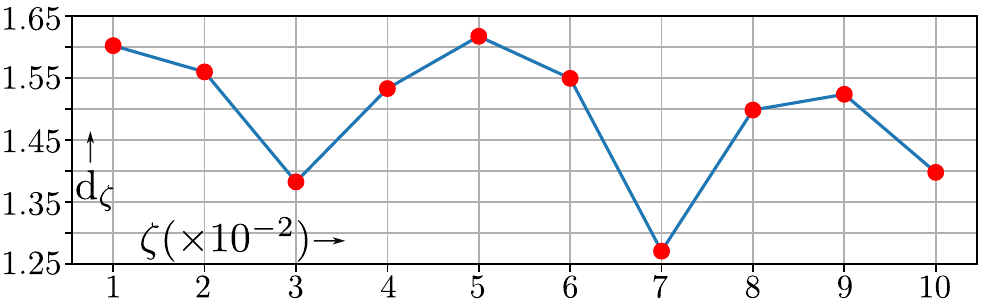}
	\captionsetup{width=\linewidth,justification=justified, singlelinecheck=false}
	\caption{Output of Algorithm \ref{alg:RCI} for different $\zeta$, simulated using models $\Theta_{\zeta}$ identified by Algorithm \ref{alg:concurrent}. Observe that for chosen $\zeta$, we obtain $\mathrm{d}_{\zeta}$ significantly lesser than $\mathrm{d}_{\mathrm{seq}}=22.6924$ and $\mathrm{d}_{\mathrm{base}}=5.0936$.}
	\label{fig:zeta_variation}
\end{figure}
\begin{figure}[t]
	\centering
	\includegraphics[width=0.85\linewidth, clip=true]{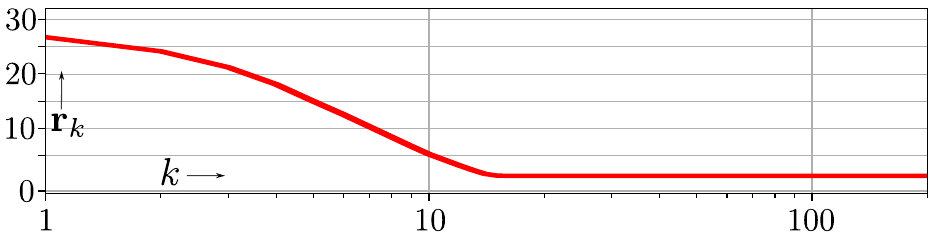}
	\captionsetup{width=\linewidth,justification=justified, singlelinecheck=false}
	\caption{Iterations of Algorithm \ref{alg:RCI} for $\Theta=\Theta_{0.07}$ and $\zeta=0.07.$}
	\label{fig:alg_1}
\end{figure}
\begin{figure}
	\centering
	\includegraphics[width=0.88\linewidth, clip=true]{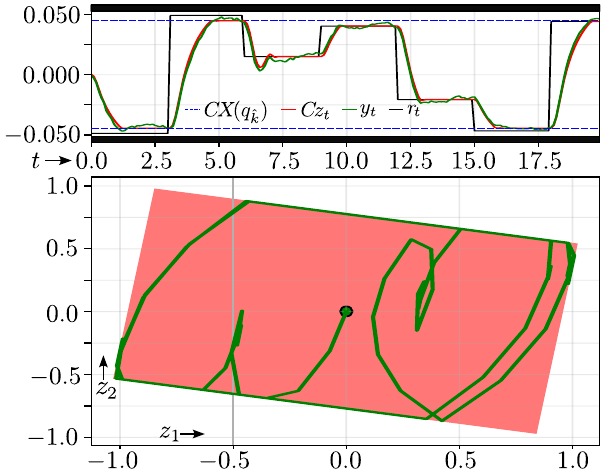}
	\captionsetup{width=\linewidth,justification=justified, singlelinecheck=false}
	\caption{Closed-loop trajectories using tracking controller \eqref{eq:tracking_controller}. The black region in top figure denotes boundaries of $Y$.}
	\label{fig:tracking_controller}
\end{figure}
\subsubsection{Controller synthesis} We use the parameters $\Theta_{0.07}$ for controller synthesis. While any robust controller that regulates \eqref{eq:nonlinear_RCI} (or \eqref{eq:uncertain_tighter}) in $X(q_{\hat{k}})$ can be used, we consider a simple tracking controller formulated as the QP
\begin{align}
	\label{eq:tracking_controller}
	&u(z,y,r) = \arg\min_{u \in U} \ \scalemath{0.9}{\|Cz^+ - r\|_2^2} \\
	&\ \ \text{s.t.}  \ \ \scalemath{0.9}{z^+ = A(p(z))z+B(p(z))u+L(p(z))(y-Cz) \in X(q_{\hat{k}}),} \nonumber 
\end{align}
which consumes the current state $z=z_t$ of~\eqref{eq:state_observer} and current output $y=y_t$ of the plant, and output reference $r=r_t$.  For sufficiently large values of $\kappa>1$ in \eqref{eq:W_characterization}, Problem \eqref{eq:tracking_controller} is recursively feasible. A study of stability properties, which involves the synthesis of ISS-Lyapunov functions to quantify the effect of $w=y-Cx$ on the closed-loop performance, are a subject of future research. Note that $(u,z^+)$ can be penalized for controller tuning.
In Figure~\ref{fig:tracking_controller}, we plot closed-loop trajectories with piecewise constant references obtained by solving~\eqref{eq:tracking_controller}. In Figure~\ref{fig:tracking_controller} (top plot), trajectories in the output space are plotted. Observe that the plant output attempts to track the reference signal $\mathrm{r}_t$ while satisfying $y_t \in Y$. Also plotted are the model output $Cz_t \in CX(q_{\hat{k}})$. In Figure~\ref{fig:tracking_controller} (bottom plot), the RCI set $X(q_{\hat{k}})$ is plotted,  with the state trajectory $z_t \in X(q_{\hat{k}})$ (in green) of the system in~\eqref{eq:state_observer}.

\section{Conclusions}
\label{sec:conclusions}
We extended the concurrent synthesis approach of \cite{Mulagaleti2024qlpv} to identify qLPV models with control synthesis guarantees through the introduction of a novel control-oriented regularization function. In Proposition \ref{prop:BP_result}, we derived conditions on existence of RCI sets based on a linear system with tightened multiplicative uncertainty, and in Corollary \ref{corr:RCI_proposition_new}, we derived a new set of configuration-constrained RCI set parameters that we used to formulate the new regularization function. We then developed Algorithm \ref{alg:RCI} to evaluate the function. Our numerical example demonstrates reduced conservativeness compared to benchmark approaches. Future research will focus on $a$) analyzing Algorithm \ref{alg:RCI}; $b$) efficient approaches to estimate $a$ verifying \eqref{eq:a_lb_2} avoiding monotonicity assumptions; and $c$) using the framework to for real-world systems.

\bibliographystyle{ieeetr} 
\bibliography{manuscript_references} 

\end{document}